\documentclass[twoside,11pt,reqno]{amsart}
\usepackage{amsmath,amssymb,amscd,mathrsfs,amscd}
\usepackage{graphics,verbatim}
\usepackage{todonotes}
\usepackage{enumitem}
\usepackage{hyperref}

\newcommand{\losemi}{{\otimes \kern -.78em \ltimes}}
\newcommand{\rosemi}{{\otimes \kern -.78em \rtimes}}

\newcommand{\Hom}{\ensuremath{\operatorname{Hom}}}

\newcommand{\Ext}{\operatorname{Ext}}

\newcommand{\la}{\lambda}

\newcommand{\St}{\operatorname{St}}

\makeatletter
\newcommand{\leqnomode}{\tagsleft@true}
\newcommand{\reqnomode}{\tagsleft@false}
\makeatother

\newtheorem{theorem}{Theorem}[subsection]

\makeatletter\let\c@fact\c@theorem\makeatother

\makeatletter\let\c@note\c@theorem\makeatother

\makeatletter\let\c@lemma\c@theorem\makeatother

\makeatletter\let\c@lemma\c@theorem\makeatother

\makeatletter\let\c@alg\c@theorem\makeatother

\makeatletter\let\c@prop\c@theorem\makeatother

\makeatletter\let\c@conj\c@theorem\makeatother

\makeatletter\let\c@cor\c@theorem\makeatother

\makeatletter\let\c@defn\c@theorem\makeatother

\theoremstyle{definition}

\makeatletter\let\c@remark\c@theorem\makeatother

\makeatletter\let\c@example\c@theorem\makeatother
\numberwithin{equation}{subsection}

\usepackage[capitalise]{cleveref}
\crefname{theorem}{Theorem}{Theorems}
\crefname{fact}{Fact}{Facts}
\crefname{note}{Note}{Notes}
\crefname{lemma}{Lemma}{Lemmas}
\crefname{alg}{Algorithm}{Algorithms}
\crefname{remark}{Remark}{Remarks}
\crefname{example}{Example}{Examples}
\crefname{prop}{Proposition}{Propositions}
\crefname{conj}{Conjecture}{Conjectures}
\crefname{cor}{Corollary}{Corollaries}
\crefname{defn}{Definition}{Definitions}
\crefname{equation}{\!\!}{\!\!} 

\newcounter{listequation}

\begin{document}

\title{On character formulas for simple and tilting modules}

\author{Paul Sobaje}
\address{Department of Mathematics \\
          Georgia Southern University}
\email{psobaje@georgiasouthern.edu}
\date{\today}

\begin{abstract}
We show that the characters of tilting modules can be used, in a concrete and explicit way, to obtain the simple characters of a connected reductive algebraic group $G$ over an algebraically closed field $\Bbbk$ of characteristic $p$, for all $p$.  Thus, once a formula for the characters of the indecomposable tilting $G$-modules has been found, a formula for the simple modules has been also.  An immediate implication is that the work of Achar, Makisumi, Riche, and Williamson in \cite{AMRW} provides a character formula for simple $G$-modules when $p>h$, the Coxeter number of $G$.
\end{abstract}

\maketitle
\vspace{-0.2in}

\section{Introduction}

Pioneering work by Riche and Williamson \cite{RW}, completed by Achar, Makisumi, Riche, and Williamson \cite{AMRW}, provides a character formula for indecomposable tilting modules of a connected reductive algebraic group $G$ over an algebraically closed field $\Bbbk$ of positive characteristic $p$, assuming that $p > h$, the Coxeter number of $G$.  The authors of \cite{AMRW} note that the characters of tilting modules take on additional significance as they can be used to compute the characters of the simple $G$-modules when $p \ge 2h-2$, thanks to an argument by Andersen \cite{A}.  Kildetoft has given an algorithm that can be implemented to compute the characters of simple modules from those of tilting modules \cite{K}, and his method does not require a bound on $p$, but only that Donkin's Tilting Module Conjecture is true (which is known to be the case when $p \ge 2h-2$ \cite{D}).  It has recently been shown, however, that this conjecture does not hold in all characteristics (specifically, it fails for the simple algebraic group of type $G_2$ in characteristic $2$ \cite{BNPS}).

In this article we give a concrete formula for the characters of the simple $G$-modules in terms of the characters of indecomposable tilting modules.  Our method works for all $p$.   Consequently, as soon as a formula for characters of tilting modules has been obtained for a given $G$ and $\Bbbk$, a character formula for the simple modules has been found also \footnote{Here we follow the standard use of the term ``character formula,'' which seems to include any recursive algorithm for obtaining characters that involves computing Kazhdan-Lusztig polynomials, $p$-Kazhdan-Lusztig polynomials, and performing basic operations in the group ring $\mathbb{Z}[X(T)]$.}.  In particular, by the work in \cite{AMRW} we have a character formula for simple $G$-modules when $p>h$.

The key step in our method is a new proof of a theorem by Parshall and Scott \cite{PS} that says that every projective indecomposable $G_1$-module has a ``stable lifting'' to a tilting module for $G$.  This proof yields an upper bound on how high one must go to find the stabilizing range, which provides an explicit, finite set of tilting modules whose characters can be used to compute the characters of the simple $G$-modules.  From here, one immediately sees that the tilting characters yield the characters of the $G_1T$-PIMs (moreover, this observation shows that knowledge of the characters of tilting modules is sufficient to determine when Donkin's Tilting Module Conjecture holds).  One then deduces the characters of the simple modules recursively from those of the $G_1T$-PIMs via Brauer-Humphreys reciprocity.

Two clarifying remarks about our work should be made at this point.  First, the formula for tilting characters in \cite{AMRW} involves recursively computing $p$-Kazhdan-Lusztig polynomials, which the authors note is ``much more involved than the original Kazhdan-Lusztig algorithm.''  Moreover, our method uses characters of indecomposable tilting modules that have highest weight above the $r$th Steinberg weight, for an explicit $r$ depending on $G$ and $p$.  In view of this, the results in this paper are at present primarily useful from a conceptual standpoint, rather than as a practical way to compute the characters of simple modules.  

Second, for certain reductive groups, an explicit connection of a different sort between tilting modules and simple modules has been known since Donkin's original paper on tilting modules for algebraic groups \cite{D}.  In the final section we explain how these results differ from those in this paper.

\subsection{Acknowledgements} I wish to thank Cornelius Pillen for communicating a significantly better proof of Theorem \ref{T:bound}, and Henning H. Andersen for the nice argument on how to use Brauer-Humphreys reciprocity to obtain simple characters.  I am grateful to Henning and to Simon Riche for careful readings and detailed comments on previous versions of this paper that have greatly improved the exposition, and to Stephen Donkin for helpful discussions about Ringel duality for truncated categories of $G$-modules.  Finally, I am indebted to Christopher P. Bendel, Daniel K. Nakano, and Cornelius Pillen for ongoing fruitful collaboration on problems pertaining to tilting modules, without which the connections in this paper would likely not have been made.

\section{Notation and Conventions}
All notation not detailed here, or later on, follows that in \cite{rags}.  We assume that $G$ is simple and simply connected (one extends these results to connected reductive groups in a standard way).  We fix a maximal torus and negative Borel subgroup $T \le B \le G$.  The character group of $T$ is $X(T)$, with $X(T)_+$ and $X_r(T)$ the sets of dominant and $p^r$-restricted weights respectively.  The group ring of characters is $\mathbb{Z}[X(T)]$.  For each $\lambda \in X(T)_+$ there is a simple module $L(\la)$, a Weyl module $\Delta(\la)$, an induced module $\nabla(\la)$, and an indecomposable tilting module $T(\la)$.  The half-sum of the positive roots is $\rho$, $\alpha_0$ is the maximal short root, and $w_0$ is the longest element of the Weyl group $W$.  The $r$th Steinberg module is the module $\St_r = L((p^r-1)\rho)$.  We denote the character of $\St_r$ by $\chi((p^r-1)\rho)$.

The $r$th Frobenius kernel of $G$ is denoted $G_r$, while $G_rT$ is the subgroup scheme of $G$ generated by $G_r$ and $T$.  For each $\la \in X(T)$ there is a simple $G_rT$-module $\widehat{L}_r(\la)$ having projective cover $\widehat{Q}_r(\la)$.  Over $G_r$ these modules are denoted as $L_r(\la)$ and $Q_r(\la)$ respectively.  The module $\widehat{Z}_r(\lambda)$ is the $r$th baby Verma module of highest weight $\lambda$.  Its character is a translate of $\chi((p^r-1)\rho)$, as
\begin{equation}\label{bVchar}
\textup{ch}(\widehat{Z}_r(\lambda))=\chi((p^r-1)\rho)e(\mu-(p^r-1)\rho).
\end{equation}

\section{The Parshall-Scott stable tilting module theorem}

In \cite{PS}, Parshall and Scott proved that a stable version of Donkin's tilting module conjecture always holds.

\begin{theorem}\cite[Theorem A.2]{PS}\label{T:P-S}
For $\Bbbk$ and $G$ fixed, there exists an integer $N \ge 0$ such that for each $\lambda \in X_1(T)$ and $n \ge N$ there is an isomorphism of $G_1T$-modules
$$T((2(p-1)\rho+w_0\lambda)+p(p^n-1)\rho) \cong \widehat{Q}_1(\lambda) \otimes {\St_n}^{(1)}.$$
\end{theorem}

As the authors of \cite{PS} note, the smallest possible choice, $N=0$, holds exactly when Donkin's Tilting Module Conjecture holds.  As stated above, this does not hold in general \cite{BNPS}.

It is clear that one can compute the characters of the $G_1T$-PIMs from the characters of the tilting modules appearing in this theorem, provided that one also knows a number $n$ that is within this stability range (so that an actual computation can be made).  The next theorem gives an upper bound for $N$, though the smallest value of $N$ could be less.  We thank C. Pillen for communicating to us the following argument, which shortens our original proof of this bound.

\begin{theorem}\label{T:bound}
Any $N$ such that
$$p^N + p^{N-1} + \cdots + 1 = \frac{p^{N+1}-1}{p-1} \ge h-1$$
will work as a bound in the previous theorem.  In particular if $p \ge h-2$, then $N=1$ suffices.
\end{theorem}

\begin{proof}
Let $N$ be fixed such that $p^{N+1}-1 \ge (p-1)(h-1)$, and set $n=N+1$.  Let $\la \in X_1(T)$, and suppose that $\mu \in X(T)_+$ and $\mu \le (p^n-1)\rho + \la$.  We have
$$\Ext_G^1(L(\mu),\St_n \otimes L(\la)) \cong \Ext_G^1(\St_n, L(\mu)^* \otimes L(\la))$$
by \cite[I.4.4]{rags} and the self-duality of $\St_n$.  Applying \cite[II.10.4]{rags}, the only composition factors of $L(\mu)^* \otimes L(\la)$ that can be extended by $\St_n$ must have the form $\St_n \otimes L(\gamma)^{(n)}$.  By highest weight considerations, since $\mu \le (p^n-1)\rho+\la$, any such $\gamma$ must satisfy
$$(p^n-1)\rho+p^n\gamma \le (p^n-1)\rho-w_0\la + \la$$
so
$$p^n \langle \gamma, \alpha_0^{\vee} \rangle = \langle p^n\gamma, \alpha_0^{\vee} \rangle \le \langle -w_0\la + \la, \alpha_0^{\vee} \rangle \le 2(p-1)(h-1).$$
Since $p^n > (p-1)(h-1)$, we have $\langle \gamma, \alpha_0^{\vee} \rangle < 2$, so that either $\gamma=0$ or $\gamma$ is minuscule.  In either case, $\Ext_G^1(\St_n, \St_n \otimes L(\gamma)^{(n)})=0$.

This proves that $\Ext_G^1(L(\mu),\St_n \otimes L(\la))=0$ for all $\mu \le (p^n-1)\rho+\la$, so that $\St_n \otimes L(\la)$ is an injective object in the truncated category $\mathcal{C}(\pi)$ of all $G$-modules whose composition factors have highest weight in the set
$$\pi = \{\sigma \mid \sigma \le (p^n-1)\rho + \la \} \subseteq X(T)_+.$$
Any indecomposable $G$-summand $M$ of $\St_n \otimes L(\la)$ is then an injective indecomposable object in $\mathcal{C}(\pi)$, and so has a simple $G$-socle (this holds in the general setting of highest weight categories \cite{CPS}, and can also be seen directly here as a consequence of \cite[A.1(4)]{rags}).  Further, the submodule $\text{soc}_{G_n}(M)$ is a direct sum of its isotypic components as a $G$-module (see \cite[II.3.16(1)]{rags}), and since each component has a nonzero $G$-socle, the simplicity of $\text{soc}_G(M)$ implies that there can only be one such component.  Now
$$\Hom_{G_n}(L((p^n-1)\rho+w_0\la),\St_n \otimes L(\la)) \cong \Hom_{G}(L((p^n-1)\rho+w_0\la),\St_n \otimes L(\la)) \cong k$$
by \cite[II.10.15]{rags}, therefore some $G$-summand of $\St_n \otimes L(\la)$ must have a simple $G_n$-socle that is isomorphic to $L((p^n-1)\rho+w_0\la)$.  This summand is projective over $G_nT$, so is isomorphic to $\widehat{Q}_n((p^n-1)\rho+w_0\la)$, and it contains the one-dimensional subspace of highest weight $(p^n-1)\rho + \la$ by \cite[Lemma II.11.6]{rags}.  By \cite{P}, $T((p^n-1)\rho + \la)$ is also a $G$-summand of $\St_n \otimes L(\la)$, so we conclude that $T((p^n-1)\rho + \la)$ is isomorphic to $\widehat{Q}_n((p^n-1)\rho+w_0\la)$ over $G_nT$.

Finally,
$$(p^n-1)\rho + w_0\la = [(p-1)\rho+w_0\la] + p(p^{n-1}-1)\rho,$$
so $L((p^n-1)\rho + w_0\la) \cong L((p-1)\rho+w_0\la) \otimes {\St_{n-1}}^{(1)}$, therefore
$$L((p-1)\rho+w_0\la) \otimes {\St_{n-1}}^{(1)} \subseteq \text{soc}_{G_1}(T((p^n-1)\rho + \la)).$$
Again appealing to \cite[II.3.16(1)]{rags}, the simplicity of the $G$-socle of $T((p^n-1)\rho + \la)$ implies that its $G_1$-socle consists of a single isotypic component.  Since the $G_1$-socle of $\St \otimes L(\la)$ has $L((p-1)\rho + w_0\la)$ occurring with multiplicity $1$ (by \cite[II.10.15]{rags}), then
$$\text{soc}_{G_1}(\St_n \otimes L(\la)) \cong \text{soc}_{G_1}(\St \otimes L(\la)) \otimes {\St_{n-1}}^{(1)}$$
has $L((p-1)\rho + w_0\la)$ occurring with multiplicity equal to the dimension of ${\St_{n-1}}^{(1)}$.  Thus
$$L((p-1)\rho+w_0\la) \otimes {\St_{n-1}}^{(1)} \cong \text{soc}_{G_1}(T((p^n-1)\rho + \la)).$$
We then have over $G_1T$ an isomorphism
\begin{align*}
T((p^n-1)\rho + \la) & = T((2(p-1)\rho + w_0((p-1)\rho+w_0\la) + p(p^{n-1}-1)\rho)\\ 
& \cong \widehat{Q}_1((p-1)\rho+w_0\la) \otimes {\St_{n-1}}^{(1)}.
\end{align*}
Recalling that $N=n-1$, and setting $\mu=(p-1)\rho+w_0\la \in X_1(T)$, this gives
$$T((2(p-1)\rho+w_0\mu) + p(p^N-1)\rho) \cong \widehat{Q}_1(\mu) \otimes {\St_N}^{(1)},$$
matching the statement in Theorem \ref{T:P-S}.
\end{proof}

\section{An explicit formula for simple characters}

Set $N=\left \lceil{\log_p(h-2)}\right \rceil$, where $h$ is the Coxeter number of $G$ and $\left \lceil{x}\right \rceil$ denotes the ceiling function (we assume that $h > 2$).  Given a character formula for tilting modules, one computes the characters of the modules
$$\{ T((2(p-1)\rho+w_0\lambda)+p(p^N-1)\rho), \; \lambda \in X_1(T)\}.$$
These tilting modules, being injective over $G_{N+1}$, have filtrations by baby Verma modules of the form $\widehat{Z}_{N+1}(\mu)$.  There are therefore coefficients $b_{\mu,\lambda} \ge 0$ such that in the Grothendieck group of $G_{N+1}T$-modules,
$$[T((2(p-1)\rho+w_0\lambda)+p(p^N-1)\rho)] = \sum b_{\mu,\lambda} [\widehat{Z}_{N+1}(\mu)].$$
Using Equation \ref{bVchar}, this is equivalent to the character equation
$$\textup{ch}(T((2(p-1)\rho+w_0\lambda)+p(p^N-1)\rho)) = \sum b_{\mu,\lambda} \,\chi((p^{N+1}-1)\rho)e(\mu-(p^{N+1}-1)\rho).$$
Since the characters on both sides are known, one can directly compute all $b_{\mu,\lambda}$.

There are also coefficients $a_{\mu,\lambda}$ such that in the Grothendieck group of $G_1T$-modules
$$[\widehat{Q}_1(\lambda)] = \sum a_{\mu,\lambda} [\widehat{Z}_1(\mu)].$$
By Theorems \ref{T:P-S} and \ref{T:bound}, the choice of $N$ guarantees that
$$T((2(p-1)\rho+w_0\lambda)+p(p^N-1)\rho) \cong \widehat{Q}_1(\lambda) \otimes {\St_N}^{(1)}.$$
Since
\begin{align*}
\text{ch}(\widehat{Z}_1(\mu))\text{ch}({\St_N}^{(1)}) & = \text{ch}(\widehat{Z}_{N+1}(\mu))e(p(p^N-1)\rho)\\
& =\text{ch}(\widehat{Z}_{N+1}(\mu+p(p^N-1)\rho)),
\end{align*}
it follows that
$$a_{\mu,\lambda}=b_{\mu+p(p^N-1)\rho,\lambda}.$$
In summary, given $\lambda \in X_1(T)$ and $\mu \in X(T)$, the multiplicity of $\widehat{Z}_1(\mu)$ in $\widehat{Q}_1(\lambda)$ is equal to $b_{\mu+p(p^N-1)\rho,\; \lambda}$, and the computation of latter is explicitly described above.

From here, the baby Verma multiplicities of all $G_1T$-PIMs are determined since
$$\text{ch}(\widehat{Q}_1(\lambda+p\mu))=\text{ch}(\widehat{Q}_1(\lambda))\cdot e(p\mu).$$
According to the reciprocity \cite[II.11.4]{rags} for $G_1T$-modules, 
$$a_{\mu,\lambda}=[\widehat{Z}_1(\mu):\widehat{L}_1(\la)].$$  These composition multiplicities determine the characters of the simple $G_1T$-modules, but one must use a careful argument since the baby Verma modules do not form a basis in the Grothendieck group of $G_1T$-modules (they only form a basis for the subgroup of all modules that are projective over $B_1$).  H. H. Andersen has communicated to us multiple arguments on how to make this last connection, including the following short proof sketch.

Assume that all the multiplicities $[\widehat{Z}_1(\la):\widehat{L}_1(\mu)]$ are known.  One finds the dimensions of the weight spaces $\widehat{L}_1(\lambda)_\nu$ by induction on $\text{ht}(\lambda - \nu)$ as follows:  First, $\widehat{L}_1(\lambda)_\lambda$ is 1-dimensional.  Assume now that we know the answer for all $\lambda, \nu$ with  $\text{ht}(\lambda - \nu) < n$ and consider $\lambda, \nu$ with $\text{ht}(\lambda -\nu) = n$. As 
$$\text{ch}(\widehat{L}_1(\lambda)) = \text{ch}(\widehat{Z}_1(\lambda)) - \sum_{\mu < \lambda} [\widehat{Z}_1(\lambda):\widehat{L}_1(\mu)] \text{ch}(\widehat{L}_1(\mu)),$$
we get
$$\dim \widehat{L}_1(\lambda)_\nu  = \dim \widehat{Z}_1(\lambda)_\nu - \sum_{\mu < \lambda} [\widehat{Z}_1(\lambda):\widehat{L}_1(\mu)] \dim \widehat{L}_1(\mu)_\nu.$$
The first term on the RHS is known and the other terms are known by induction (since $\mu < \la$ implies that $\text{ht}(\mu-\nu)<n$).  Note that all sums above are finite, and because of the identity $\text{ch}(\widehat{L}_1(\la + p\mu))=\text{ch}(\widehat{L}_1(\la))e(p\mu)$, the whole process is finite.

\section{Comparison with Ringel Duality}

In \cite[Proposition 3.10]{D}, Donkin uses Ringel duality (in a nontrivial way) to give an explicit relationship between tilting characters and simple characters of general linear groups.  This result was later generalized to special orthogonal and symplectic groups by work of Adamovich and Rybnikov in \cite{AR}.   Thus, for connected reductive groups of classical type (by which we mean that the components of the root system are all of types $A$-$D$), it was known that tilting characters determine simple characters (and vice versa even).  We should therefore clarify at this point the difference between these results and those found in this paper.

First, the method presented here works for all reductive groups, whereas it is not currently known how to generalize Donkin's method to include simple groups of exceptional type.  A second important distinction is that in these other papers the simple characters for a given connected reductive group $G$ of classical type are determined by the tilting characters of some {\em other} connected reductive group $H$ of classical type that is in general different from $G$.  This turns out to be especially significant when trying to apply the work in \cite{AMRW}, which requires that $p>h$, the Coxeter number of $G$.

A very short summary of Donkin's result can be found in \cite[E.10]{rags}.  From this, one can work out examples of a $GL_n$ and $p$, with $n<p$, for which the characters of simple $GL_n$-modules of restricted highest weight will be given, by \cite[Proposition 3.10]{D}, from characters of tilting modules for various $GL_m$ where $m$ is not always less than $p$.  For example, let $n=6$, $p=7$, and $\lambda=(15,11,9,4,2,0)$ (which is a restricted weight since $a_i-a_{i+1} < 7$).  The $GL_6$ composition factors of $\nabla(\lambda)$ will then be given by good filtration factors of a tilting module for $GL_{15}$.  In this sense, one cannot directly appeal to \cite[Proposition 3.10]{D} in order to show that a tilting character formula for $GL_n$ and $p>n$ implies a simple character formula for $GL_n$ under the same bound on $p$ (which is the main purpose of this article, for general connected reductive groups).

Nonetheless, the results in \cite{D} and \cite{AR} are powerful, and could ultimately hold the key to further unlocking the mystery of modular character formulas.  Moreover, it would be of great interest to know if an appropriate generalization of \cite[Proposition 3.10]{D} can be found for groups of exceptional type.

\providecommand{\bysame}{\leavevmode\hbox
to3em{\hrulefill}\thinspace}

\end{document}